\documentclass[a4paper]{article}
\usepackage[latin1]{inputenc}
\usepackage{amssymb}
\usepackage[arrow,matrix,curve]{xy}
\usepackage{graphicx,color}
\usepackage{makeidx}
\usepackage[thmmarks,standard,hyperref]{ntheorem}
\usepackage{hyperref}
\usepackage{fancyhdr}

\author{Matthias Moßburger}
\title{On a Unified Analysis in the language of preordered sets}
\theoremnumbering{arabic}
\theoremstyle{plain}
\theoremsymbol{}
\theorembodyfont{\itshape}
\theoremheaderfont{\normalfont\bfseries}
\theoremseparator{.}
\renewtheorem{Theorem}{Theorem}
\renewtheorem{Satz}[Theorem]{Satz}
\renewtheorem{Proposition}[Theorem]{Proposition}
\renewtheorem{Lemma}[Theorem]{Lemma}
\renewtheorem{Korollar}[Theorem]{Korollar}

\renewtheorem{theorem}{Theorem}
\renewtheorem{proposition}[theorem]{Proposition}
\renewtheorem{lemma}[theorem]{Lemma}
\renewtheorem{corollary}[theorem]{Corollary}

\theoremsymbol{\ensuremath{\Box}}
\theorembodyfont{\upshape}
\renewtheorem{Beispiel}[Theorem]{Beispiel}
\renewtheorem{Bemerkung}[Theorem]{Bemerkung}
\renewtheorem{Anmerkung}[Theorem]{Anmerkung}
\renewtheorem{Definition}[Theorem]{Definition}

\renewtheorem{example}[theorem]{Example}
\renewtheorem{definition}[theorem]{Definition}

\theoremstyle{nonumberplain}
\theoremheaderfont{\itshape}
\theorembodyfont{\normalfont}
\renewtheorem{Beweis}{Beweis}

\renewtheorem{proof}{Proof}

\qedsymbol{\ensuremath{_\blacksquare}}

\newcommand{\N}{\mathbb{N}}

\newcommand{\Q}{\mathbb{Q}}
\newcommand{\R}{\mathbb{R}}

\newcommand{\fin}{\stackrel{\mbox{\scriptsize fin}}{\le}}
\newcommand{\loc}{\stackrel{\mbox{\scriptsize loc}}{\le}}
\newcommand{\tang}{\stackrel{\mbox{\scriptsize tan}}{\le}}
\newcommand{\locae}{\stackrel{\mbox{\scriptsize loc}}{\sim}}
\newcommand{\tangae}{\stackrel{\mbox{\scriptsize tan}}{\sim}}
\newcommand{\fue}{\stackrel{\mu}{\le}}

\begin{document}

{\LARGE On a Unified Analysis in the language of preordered sets}

\vspace{3mm}

\begin{center}
{\large Matthias Moßburger}
\end{center}

\vspace{3mm}

\noindent
\textsc{Abstract.} 
We define a notion which contains numerous basic notions of Analysis as special cases, for example limit, continuity, differential, Riemann and Lebesgue integral, root and exponential functions. Properties like additivity or linearity of all mentioned notions follow from one single theorem. 
A generalisation of Definition 1 contains basic ideas of other theories such as topology, group and measure theory.

\section{Definition and Examples}
\label{sec:definition}

We use the following notions:
A {\em preorder} on a set $M$ is a binary relation which is reflexive and transitive.
A mapping between preordered sets is called {\em isotonic}, if it preserves the given order. Let $(M,\le)$ be a preordered set. An element $m\in M$ is called a {\em lower bound of} $A\subseteq M$, if $m\le a$ for all $a\in A$\,; in this case we write $m\le A$. Similar: {\em upper bound} and $A\le m$\,. 
$(M,\le)$ is called {\em complete}, if every nonempty subset which is bounded from below has an unique infimum and every nonempty subset which is bounded from above has an unique supremum. 
In this case, $\le$ is antisymmetric: If $a\le b \le a$\,, then $a=\inf\lbrace a,b\rbrace=b$\,.

Definition \ref{def:phiAbschluss} contains all notions mentioned at the beginning. The idea behind this definition: Let $K$ be a set of elements $a$ with an ``obvious'' value $\varphi(a)$\,; these values can be used to get a continuation of $\varphi$ through approximations from above and below. 
For example: Constant sequences have an obvious limit. In other words, we have an obvious function $\,\lim: K\longrightarrow\R$ if $K$ is the set of all constant real sequences. But how can we extend the domain of $\lim$?

\begin{definition}
\label{def:phiAbschluss}
Let $L$ and $M$ be preordered sets, $M$ complete, $K\subseteq L$ and
$\varphi:K\longrightarrow M$ isotonic. 
\[\begin{xy}
\xymatrix{
  K \ar[r]^{\subseteq} \ar[rd]_{\varphi} 
  & 
  {\overline K} \ar[r]^{\subseteq} \ar@{.>}[d]^{\varphi^*}
  & 
  L
  \\ 
  & M}
\end{xy}\]
\begin{enumerate}
\renewcommand{\theenumi}{\alph{enumi}}
\renewcommand{\labelenumi}{\rm(\theenumi)}
	\item An element $f\in L$ is called {\em K-bounded}\index{bounded}\,, if there are $a,b\in K$ with $a\le f\le b$\,. For $K$-bounded $f\in L$ let
	\[
	 \varphi_*(f) \ :=\ 
	 \sup \lbrace\, \varphi(a) \,|\, a\in K \mbox{ and } a\le f
	 \,\rbrace
	 \quad \mbox{ and } \quad
	 \varphi^*(f) \ :=\ 
	 \inf \lbrace\, \varphi(b) \,|\, b\in K \mbox{ and } f\le b
	 \,\rbrace \ .
	\]
	\item The set $\overline K$ of all $K$-bounded $f\in L$ with $\varphi_*(f)=\varphi^*(f)$ is called 
	{\em $\varphi$-closure of $K$ in $L$}\index{closure}\,. 
\end{enumerate}
\end{definition}

%
%
%

\begin{theorem}
\label{satz:phiAbschluss}
Let $K,L,M$ and $\varphi$ be as in Definition \ref{def:phiAbschluss}.
Then $\varphi^*: \overline K \longrightarrow M$ is the unique isotonic continuation of $\varphi$ on $\overline K$\,.
\end{theorem}

\begin{proof} 
$\varphi^*$ is a continuation:
Let $f\in K$. 
Since 
$\varphi(f) \in \lbrace\, \varphi(a) \,|\, a\in K \,,\, a \le f \,\rbrace \le \varphi(f)$ we have 
$\varphi(f) \le \varphi_*(f) \le \varphi(f)$\,,
so $\varphi(f)=\varphi_*(f)$\,, because $M$ is complete (so $\le$ is antisymmetric). Similar: $\varphi(f)=\varphi^*(f)$\,, so $f\in\overline K$. 
$\varphi^*$ is isotonic:
Let $f,g\in\overline K$ and $f\le g$\,. 
Then $\lbrace\, \varphi(b) \,|\, g\le b\in K \,\rbrace \subseteq 
\lbrace\, \varphi(b) \,|\, f\le b\in K \,\rbrace$\,, so 
$\varphi^*(f) \le \varphi^*(g)$\,. 
$\varphi^*$ is unique:
Let $\psi: \overline K \longrightarrow M$
be an isotonic continuation of $\varphi$ on $\overline K$\,, and $f\in\overline K$\,. 
For all $a,b\in K$ with $a\le f\le b$ we have
$\varphi(a) = \psi(a) \le \psi(f) \le \psi(b) = \varphi(b)$\,, so
$\varphi^*(f) = \varphi_*(f) \le \psi(f) \le \varphi^*(f)$\,, so
$\varphi^*(f) = \psi(f)$\,, because $M$ is complete. 
\end{proof}

Now we show that Definition \ref{def:phiAbschluss} contains all notions mentioned at the beginning:

\subsection{Limit of nets in metric spaces}
\label{subsec:limit}

Idea: Constant sequences have an obvious limit.
Let $(\Lambda,\prec)$ be a directed set (a preordered set with the additional  property that every pair of elements has an upper bound),
$L$ the set of all functions $f:\Lambda\longrightarrow\R$\,, $K$ the set of all constant $f\in L$ and 
$\varphi:K\longrightarrow \R \,,\, f \longmapsto f(n)$ for an $n\in\Lambda$\,. 
The elements of $L$ are called real nets, and $L$ is the set of all real sequences if $(\Lambda,\prec) = (\N,\le)$.

A decisive idea behind convergence is that we are not interested in the {\em whole} sequence, but only in the fact that a sequence is {\em finally} between some bounds. 
We specify this idea with a preorder {\em finally less or equal}\index{finally less} (cf. \cite{Mos}, p. 13): 
For $f,g\in L$ let
\[
 f \fin g
 \quad :\Leftrightarrow\quad
 \exists\, N\in\Lambda \ \,\forall\, n\succ N :\, f(n) \le g(n) \ .
\]
$(L,\fin)$ is a preordered set, $(\R,\le)$ complete and $\varphi$ isotonic, so all assumptions of Definition \ref{def:phiAbschluss} are fulfilled.
The $\varphi$-closure $\overline K$ is the set of all convergent $f\in L$\,, 
and $\varphi^*(f) = \lim_{\Lambda}(f)$ is the limit of $f\in\overline K$.


A generalization: 
Convergence of nets in metric spaces $(X,d)$ can be defined with the help of convergent real nets: 
A net $f: \Lambda\longrightarrow X$ has limit $x\in X$ if and only if the real net 
$n \longmapsto d(f(n),x)$ has limit $0$\,.

\subsection{Continuous mappings in metric spaces}
\label{subsec:continuous}

Idea: Constant mappings are continuous. 
Let $X$ be a topological space, $U\subseteq X$ a neighbourhood of $p\in X$\,, $L$ the set of all functions $f:U\longrightarrow\R$\,, 
$K$ the set of all constant functions $f\in L$ and 
$\varphi:K\longrightarrow \R \,,\, f \longmapsto f(p)$\,.

A decisive idea behind continuity is that we are only interested in the fact that values are {\em locally} between some bounds. 
We specify this idea with a preorder {\em locally less or equal}\index{locally less}\,: 
For $f,g\in L$\, let
\[
 f \loc g
 \quad :\Leftrightarrow\quad
 \exists \mbox{ neighbourhood $V\!$ of } p \ \ \forall\, x\in U\cap V :\, 
 f(x) \le g(x) \ .
\]
The equivalence relation \ $f \locae g \ :\Leftrightarrow\ f\loc g\loc f$ \ says that $f$ and $g$ are equal in a neighbourhood of $p$\,, 
so the equivalence classes are function-germs at $p$\,. 
$(L,\loc)$ is a preordered set, $(\R,\le)$ complete and $\varphi$ isotonic, so all assumptions of Definition \ref{def:phiAbschluss} are fulfilled.
For all $K$-bounded $f\in L$ we have 
\[
 f \mbox{ continuous at } p
\quad\Longleftrightarrow\quad
\forall\, \epsilon\in\R^+: |f - f(p)| \loc \epsilon
\quad\Longleftrightarrow\quad
\varphi_*(f) = \varphi^*(f) \ ,
\]
so $\overline K$ is the set of all $f\in L$ which are continuous at $p$\,.

A generalization: 
Continuity of mappings with values in metric spaces $(X,d)$ can be defined with the help of real functions: 
A mapping $f: U\longrightarrow X$ is continuous at $p\in U$ if and only if the real function $x \longmapsto d(f(x),f(p))$ is continuous at $p$\,.

\subsection{Limit and continuity in topological spaces}
\label{subsec:Topologie}

Idea: A filter converges if it is finer then a neighbourhood filter. 
Let $X$ be a topological space. The neighbourhood filter of $x\in X$ is denoted by ${\cal U}(x)$\,, and ${\cal P}(X)$ is the powerset of $X$. 
Let $L$ be the set of all filters in $X$, 
$K:= \lbrace\, {\cal F}\in L \ |\ \exists\, x\in X\!: \,
{\cal U}(x)\subseteq {\cal F} \,\rbrace$ and 
$\varphi: K\longrightarrow {\cal P}(X) \,,\, 
{\cal F} \longmapsto \bigcap\limits_{F\in{\cal F}} \overline{F}$.
The set $\varphi({\cal F})$ is called set of limits of ${\cal F}\in K$. 
If $X$ is a Hausdorff space and ${\cal F} \supseteq{\cal U}(x)$, then $\varphi({\cal F})=\lbrace x\rbrace$\,. 
$(L,\subseteq)$ is a preordered set, $({\cal P}(X),\supseteq)$ complete and $\varphi$ isotonic, so all assumptions of Definition \ref{def:phiAbschluss} are fulfilled. 
The mapping $\varphi$ can not be continued in the sense of Definition \ref{def:phiAbschluss}: 
All $K$-bounded ${\cal F}\in L$ are elements of $K$, so $\overline K=K$.

Continuity can also be expressed with the help of the preorder $\subseteq$\,: 
Let $Y$ be a topological space. A mapping $f:X\longrightarrow Y$ is continuous at $x\in X$ if and only if \,${\cal U}(f(x)) \subseteq f({\cal U}(x))$\,.

\subsection{Differential}
\label{subsec:Differential}

Idea: The differential of a linear function $f$\, is $f$\,. 
Let $n\in\N$\,, $U\subseteq \R^n$ a neighbourhood of $0\in\R^n$,
$L$ the set of all functions $f:U\longrightarrow\R$ with $f(0)=0$\,, 
$M$ the set of all linear functions $f:\R^n\longrightarrow\R$\,, 
$i$ the inclusion $U\hookrightarrow\R^n$, 
$K$ the set of all $f\circ i$ with $f\in M$ 
and $\varphi: K\longrightarrow M \,,\, f\circ i \longmapsto f$ (unique linear continuation on $\R^n$). 

A decisive idea behind the derivative in dimension $n=1$ could be formulated as follows. A tangent line for $f$ at 0 is a straight line $y=mx$ with a special local property (see figure):  
Every straight line $y=\overline{m}x$ with $\overline{m}>m$ is locally below the graph of $f$ for $x<0$ and locally above the graph of $f$ for $x>0$\,; and in case $\overline{m}<m$ ``below'' and ``above'' are reversed. This means that there is an $a\in K$ such that \ 
$
 a - \epsilon\cdot |id| \ \loc\ f \ \loc\ a + \epsilon\cdot |id| 
 \mbox{ for all } \epsilon\in\R^+  
$
and $\varphi(a) = df_0$\,.

\begin{figure}[htbp]
	\centering
		\includegraphics[width=0.30\textwidth, angle=-90]
		{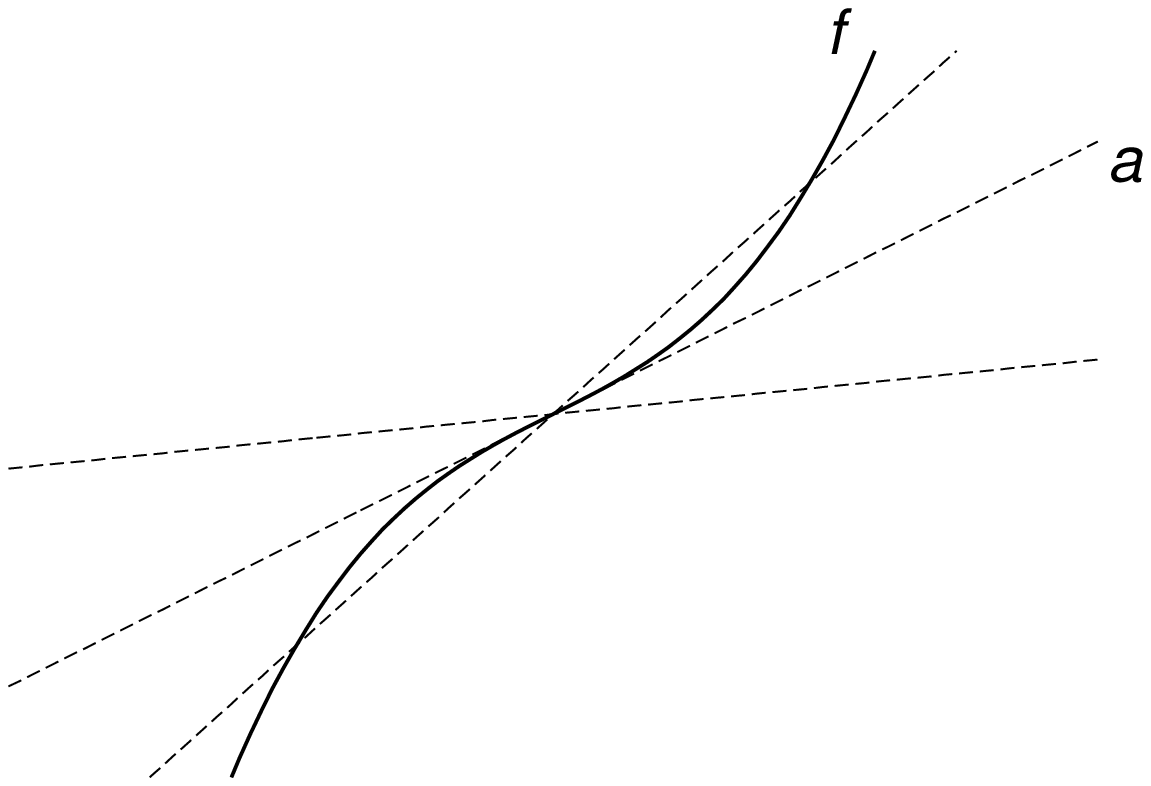}
	\label{fig:diff1} 
\end{figure}

This idea can be generalized to all $n\in\N$\,: We define a preorder {\em tangential less or equal}\index{tangential less}
\[
 f \tang g
 \quad :\Leftrightarrow\quad
 \forall\,\epsilon\in\R^+:\, 
 f \loc g + \epsilon\cdot \left\| id \right\| 
\qquad
\mbox{for } f,g\in L \mbox{ resp. } f,g\in M
\]
and the equivalence relation {\em tangential equivalent}\index{tangential equivalent} \ 
$f \tangae g \ :\Leftrightarrow\ f\tang g\tang f$\,. 
We have
\begin{enumerate}
\renewcommand{\theenumi}{\alph{enumi}}
\renewcommand{\labelenumi}{\rm(\theenumi)}
	\item $a\tang b \ \Longrightarrow\ a=b$ \quad for all $a,b\in K$ resp. $a,b\in M$  
	\item $f \tangae g 
	 \quad\Longleftrightarrow\quad 
	 \displaystyle\lim\limits_{x\rightarrow0}\frac{f(x)-g(x)}{\|x\|} 
	 =0$  
	 \qquad for all $f,g\in L$.
\end{enumerate}
It follows that every subset of $M$ which contains more than one element is unbounded with regard to $\tang$, so $M$ is complete. All assumptions of Definition \ref{def:phiAbschluss} are fulfilled, and for all $f\in L$ we have
\[
 f \mbox{ is $K$-bounded }
 \quad\Longleftrightarrow\quad 
 \exists\, a,b\in K :\, a\tang f\tang b
 \quad\Longleftrightarrow\quad 
 \exists\, a\in K :\, f \tangae a
\]
according to statement (a), so $f\in L$ is $K$-bounded if and only if its total differential $df_0$ exists. 
Furthermore, $\varphi_*(f) = \varphi(a) = \varphi^*(f)$ for $f\in L$ and $a\in K$ with $f\tangae a$\,, so $\overline K$ is the set of all $f\in L$ which are totally differentiable at $0$\,, and $\varphi^*(f)$ is the differential $df_0$\,.

\subsection{Riemann integral}
\label{subsec:Riemann}

Idea: Step functions with a finite number of steps have an obvious integral.
Let $n\in\N$\,, $\iota$ the Jordan measure on $\R^n$ (cf. \cite{Wal} 7.4) and $D\subseteq\R^n$ Jordan measurable. 
A function $f:D\longrightarrow\R$ is called a (Jordan) step function, if $f(D)$ is finite and $f^{-1}(y)$ is Jordan measurable for all $y\in f(D)$\,. 
Let $L$ be the set of all functions $f:D\longrightarrow\R$\,, $K$ the set of all step functions in $L$\,, $\varphi: K\longrightarrow\R$\,, $f\longmapsto
\sum\limits_{y\in f(D)} y\cdot \iota\big(f^{-1}(y)\big)$ 
and for all $f,g\in L$ let
$f\le g \ :\Leftrightarrow\ \forall\,x\in D:\, f(x)\le g(x)$\,. 
Then $(L,\le)$ is a preordered set, $(\R,\le)$ complete and $\varphi$ isotonic, so all assumptions of Definition \ref{def:phiAbschluss} are fulfilled.
$\overline K$ is the set of all Riemann integrable $f\in L$\,, and $\varphi^*(f)$ is the Riemann integral of $f$\,.

\subsection{Lebesgue integral}
\label{subsec:Lebesgue}

Idea: Step functions with an at most countable number of steps have an obvious integral.
Let $(X, {\cal A}, \mu)$ be a measurable space and $D\in \cal A$\,. 
A function $f:D\longrightarrow\R$ is called a $\mu$-step function, if $f(D)$ is at most countable, $f^{-1}(y)\in{\cal A}$ for all $y\in f(D)$ and 
$\varphi(f) := \sum\limits_{y\in f(D)} y\cdot \mu\big(f^{-1}(y)\big)$ converges absolutely. 
Let $L$ be the set of all functions $f:D\longrightarrow\R$\,, $K$ the set of all $\mu$-step functions in $L$ and 
$\varphi: K\longrightarrow\R$\,, $f\longmapsto\varphi(f)$\,.

A decisive idea behind the Lebesgue integral is that we are only interested in the fact that a function is {\em $\mu$-almost everywhere} between $\mu$-step functions. 
We specify this idea with a preorder {\em $\mu$-almost everywhere less or equal}\index{almost less}\,: 
For $f,g\in L$\, let
\[
f \fue g \quad :\Leftrightarrow \quad
\mu\big( \lbrace\, x\in D \,|\, f(x) > g(x) \,\rbrace \big) = 0 \ .
\]
$(L,\fue)$ is a preordered set, $(\R,\le)$ complete and $\varphi$ isotonic, so all assumptions of Definition \ref{def:phiAbschluss} are fulfilled. 
The $\varphi$-closure $\overline{K}$ is the set of all Lebesgue integrable $f\in L$\,, and $\varphi^*(f)$ is the Lebesgue integral of $f\in\overline{K}$\,:
A function $f\in L$ is Lebesgue integrable with Lebesgue integral $J(f)$ if and only if (cf. \cite{Wal} 9.7(e) and 9.32.1)
\begin{enumerate}
\renewcommand{\theenumi}{\alph{enumi}}
\renewcommand{\labelenumi}{\rm(\theenumi)}
	\item $\exists\, \pi^*\in P_D:\, S(\pi^*,|f|) < \infty$ 
	\item $J_*(f)=J^*(f)=J(f)$\,.
\end{enumerate}
$P_D$ is the set of all (measurable and countable) partitions of $D$, 
$S(\pi^*,|f|) := \sum\limits_{A\in\pi^*} \sup( |f|(A) ) \cdot \mu(A)$\,,
$J_*$ the lower and $J^*$ the upper integral. 
We have
\begin{enumerate}
\renewcommand{\theenumi}{\alph{enumi}}
\renewcommand{\labelenumi}{\rm(\theenumi)}
	\item $\exists\, \pi^*\in P_D:\, S(\pi^*,|f|) < \infty 
	 \quad\Longleftrightarrow\quad
	 f$ is $K$-bounded.
	\item $J_*(f)=\varphi_*(f)$ and $J^*(f)=\varphi^*(f)$ 
	 for all $K$-bounded $f$.
\end{enumerate}

A generalization: 
Let $V$ be a normed real vector space and $V'$ its dual space (the set of all continuous and linear functions $V\longrightarrow\R$). A mapping $f:D\longrightarrow V$ is called integrable, if
\begin{enumerate}
\renewcommand{\theenumi}{\alph{enumi}}
\renewcommand{\labelenumi}{\rm(\theenumi)}
	\item $\forall\, \alpha\in V': \alpha\circ f \in \overline{K}$ 
	\item $\exists\, v\in V \ \forall\, \alpha\in V': 
	 \alpha(v)=\varphi^*(\alpha\circ f)$\,.
\end{enumerate}
In this case $v$ is unique,
because the canonical mapping $V\longrightarrow V''$ is injective. 
We call $\int{f}d\mu := v$ the integral of $f$.

\subsection{Roots}
\label{subsec:Wurzeln}

Idea: The $n$-th root of $x^n$ is $x$. 
Let $L:=M:=\R^+$, $n\in\N$\,, 
$K:=\lbrace x^n \,|\, x\in\Q^+ \rbrace$
and $\varphi: K\longrightarrow \R^+$, $x^n\longmapsto x$\,. 
In this case, Definition\,\ref{def:phiAbschluss} formulates the idea of computing roots through nested intervalls:
For all $f\in\R^+$ we have 
\[
 \lbrace\, \varphi(a) \,|\, a\in K \mbox{ and } a\le f \,\rbrace = 
 \lbrace\, x\in\Q^+ \,|\, x^n\le f \,\rbrace \ ,
\] 
\[
 \lbrace\, \varphi(b) \,|\, b\in K \mbox{ and } f\le b \,\rbrace = 
 \lbrace\, y\in\Q^+ \,|\, f\le y^n \,\rbrace \ ,
\]
so $\varphi_*(f) = \varphi^*(f) = \sqrt[n]{f\,}$
and $\overline{K}=\R^+$.

\subsection{Exponential functions}
\label{subsec:Exponential}

Let $K:=\Q$\,, $L:= \R$\,, $M:=\R^+$, $u>1$
and $\varphi: \Q\longrightarrow \R^+ ,\, r\longmapsto u^r$. 
For all $f\in\R$ we have 
\[
 \varphi_*(f) = 
 \sup\lbrace\, u^a \,|\, a\in \Q \mbox{ and } a\le f \,\rbrace = 
 \inf\lbrace\, u^b \,|\, b\in \Q \mbox{ and } f\le b \,\rbrace = 
 \varphi^*(f) = u^f \ .
\]
Especially: $\R$ is the $\varphi$-closure of $\Q$\,.


\section{Algebraic properties}
\label{sec:Algebra}

Now we show that properties like additivity or linearity of all notions mentioned  at the beginning follow from one single theorem (after one has realized that the corresponding $K,L,M$ are special cases of Definition \ref{def:Gruppe}).

\begin{definition}
\label{def:Gruppe}
A {\em preordered group}\,\index{group} is a group $(G,+)$ with a {\em compatible}\, preorder $\le$\,, which means that for all $a,b,c,d\in G$ we have
\begin{equation}
\label{eq:Gruppe}
 a\le b \mbox{ and } c\le d  \quad\Longrightarrow\quad a+c \le b+d \ .
\end{equation}
A {\em preordered real vector space}\,\index{vector space} is a real vector space $V$ with a {\em compatible}\, preorder $\le$\,, which means that in addition to (\ref{eq:Gruppe}) we have for all $a,b\in V$:
\[
 a\le b \mbox{ and } \lambda\in\R^+   \quad\Longrightarrow\quad 
\lambda a \le \lambda b \ .
\]
\end{definition}

\begin{example}
\label{bsp:group}
\begin{enumerate}
\renewcommand{\theenumi}{\alph{enumi}}
\renewcommand{\labelenumi}{\rm(\theenumi)}
	\item Except for section \ref{subsec:Topologie} (the question of additivity does not arise there), we have in all sections from \ref{subsec:limit} to \ref{subsec:Exponential}: $L$ and $M$ are preordered groups, $K$ is a subgroup of $L$ and $\varphi$ is an isotonic group homomorphism.
	\item In sections \ref{subsec:limit}, \ref{subsec:continuous}, \ref{subsec:Differential}, \ref{subsec:Riemann} and \ref{subsec:Lebesgue} we have: $L$ and $M$ are preordered vector spaces, $K$ is a subspace of $L$ and $\varphi$ is isotonic and linear.
\end{enumerate}
\end{example}

In every complete preordered group $(M,+,\le)$, \,$\inf$ and \,$\sup$ are compatible with $+$\,, which means that for all nonempty and bounded $A,B\subseteq M$ we have
\[
 \sup(-A) = -\inf(A)
 \quad , \quad
 \inf(A+B) = \inf(A) + \inf(B)
 \quad\mbox{ and }\quad
 \sup(A+B) = \sup(A) + \sup(B) \, .
\]
In every complete preordered real vector space $M$ we have for all $\lambda\in\R^+$\,:
\[
 \sup(\lambda A) = \lambda \sup(A)
\qquad\mbox{and}\qquad
 \inf(\lambda A) = \lambda \inf(A) \ .
\]

\begin{theorem}
\label{satz:Hom}
\begin{enumerate}
\renewcommand{\theenumi}{\alph{enumi}}
\renewcommand{\labelenumi}{\rm(\theenumi)}
	\item Let $L$ and $M$ be preordered groups, $M$  
	 complete, $K$ a subgroup of $L$ and
   $\varphi:K\longrightarrow M$ an isotonic group homomorphism. 
	 Then $\varphi^*: \overline{K}\longrightarrow M$ is an isotonic group homomorphism. 
	\item Let $L$ and $M$ be preordered real vector spaces, $M$  
	 complete, $K$ a subspace of $L$\,,
   $\varphi:K\longrightarrow M$ isotonic and linear. 
	 Then $\varphi^*: \overline{K}\longrightarrow M$ is isotonic and linear.
\end{enumerate}
\end{theorem}
\[\begin{xy}
\xymatrix{
  K \ar[r]^{\subseteq} \ar[rd]_{\varphi} 
  & 
  {\overline K} \ar[r]^{\subseteq} \ar@{.>}[d]^{\varphi^*}
  & 
  L
  \\ 
  & M}
\end{xy}\]
\begin{proof}
(a) Let $f,g\in \overline{K}$. Then $f+g$ is $K$-bounded and
\[
 \lbrace\, \varphi(a_1) \,|\, a_1\in K , a_1\le f \,\rbrace
 +
 \lbrace\, \varphi(a_2) \,|\, a_2\in K , a_2\le g \,\rbrace
 \ \subseteq\ 
 \lbrace\, \varphi(a) \,|\, a\in K , a\le f+g \,\rbrace \ ,
\]
so $\varphi_*(f) + \varphi_*(g) \le \varphi_*(f+g)$\,. 
Similar: $\varphi^*(f+g) \le \varphi^*(f) + \varphi^*(g)$\,, so
\[
 \varphi_*(f) + \varphi_*(g) \le \varphi_*(f+g) 
\le \varphi^*(f+g) \le \varphi^*(f) + \varphi^*(g) 
= \varphi_*(f) + \varphi_*(g) \ ,
\]
and so $\varphi_*(f+g) = \varphi^*(f+g) = \varphi^*(f) + \varphi^*(g)$\,, 
because $M$ is complete (and $\le$ antisymmetric).
It remains to show that $\varphi_*(-f) = \varphi^*(-f)$\,. 
The inverse $-f$ is also $K$-bounded, and
\[
 \lbrace\, \varphi(a) \,|\, a\in K , a\le -f \,\rbrace
\ =\ 
-\lbrace\, \varphi(-a) \,|\, a\in K , f\le -a \,\rbrace
\ = \ 
-\lbrace\, \varphi(b) \,|\, b\in K , f\le b \,\rbrace \ ,
\]
so $\varphi_*(-f) = -\varphi^*(f)$ and $\varphi^*(-f) = -\varphi_*(f)$\,, which implies $\varphi_*(-f) = \varphi^*(-f)$\,.
\\[7pt]
(b) Let $f\in \overline K$ and $\lambda\in\R^+$.
Then $\lambda f$ is $K$-bounded and
\[
 \lbrace\, \varphi(a) \,|\, a\in K , a\le \lambda f \,\rbrace
 =
 \lambda\lbrace\, \varphi(\lambda^{-1}a) \,|\, a\in K , \lambda^{-1} a\le f \,\rbrace
 =
 \lambda\lbrace\, \varphi(a') \,|\, a'\in K , a'\le f \,\rbrace \ ,
\]
so $\varphi_*(\lambda f) = \lambda \varphi_*(f)$\,.
Similar: $\varphi^*(\lambda f) = \lambda \varphi^*(f)$\,, so 
$\varphi_*(\lambda f) = \varphi^*(\lambda f)$\,. 
\end{proof}

It follows that $\varphi^*$ is a group homomorphism resp. linear in all examples of  \ref{bsp:group}(a) resp. (b). 
Furthermore, $\varphi^*$ is a ring homomorphism in sections \ref{subsec:limit} and \ref{subsec:continuous} :

\begin{corollary}
\label{kor:Hom}
	In addition to the assumptions of Theorem \ref{satz:Hom}(a) let $L$ and $M$ be rings, $K$ a subring of $L$\,, $K+\ker\varphi^*=\overline K$\,, $f\cdot(\ker\varphi^*)\subseteq \ker\varphi^*$ and $(\ker\varphi^*)\cdot f \subseteq \ker\varphi^*$ for all $K$-bounded $f\in L$\,, and $\varphi:K\longrightarrow M$ an isotonic ring homomorphism. Then $\varphi^*: \overline{K}\longrightarrow M$ is an isotonic ring homomorphism.
	\hfill $\Box$
\end{corollary}

The proof that $\ker \varphi^*$ in sections \ref{subsec:limit} and \ref{subsec:continuous} fulfills the assumptions of Corollary \ref{kor:Hom}  
can be formulated in the language of preordered sets in a quite compact way:
Let $K,L,M$ and $\varphi$ be as in sections \ref{subsec:limit} resp. \ref{subsec:continuous}. The equation $K+\ker\varphi^*=\overline K$ is obvious. 
Let $g\in\ker\varphi^*$ and $f\in L$ be $K$-bounded. Then $fg\in\ker\varphi^*$:
\\
Let $\epsilon\in K$ and $\varphi(\epsilon)>0$\,. 
By assumption there is an element $b\in K$ with $\varphi(b)>0$ and $-b \le f \le b$ (finally resp. locally). 
Since $g\in\ker\varphi^*$ we have  
$-\frac{\epsilon}{b} \le g \le \frac{\epsilon}{b}$\,, 
so $-\epsilon \le fg \le \epsilon$\,.


\section{Generalisation}
\label{sec:generalisation}

So far we considered notions which are defined through approximations from above and below. Now we generalize Definition \ref{def:phiAbschluss} since many notions are only defined through approximations from above. 
A preordered set $(M,\le)$ is called {\em complete from below}\index{complete}\,, if every nonempty subset has an unique infimum. In this case, $\le$ is antisymmetric.

\[\begin{xy}
\xymatrix{
  K \ar[r]^{\subseteq} \ar[rd]_{\varphi} 
  & 
  {K^*} \ar[r]^{\subseteq} \ar@{.>}[d]^{\varphi^*}
  & 
  L
	&
	&
	{\cal L} \ar[r]^{\subseteq} \ar[rd]_{id_{\cal L}} 
  & 
  {\cal L}^* \ar[r]^{\subseteq} \ar@{.>}[d]^{id_{\cal L}^{\,*}}
  & 
  {\cal P}(L)
	&
  \\ 
  & 
	M
	&&&&
	{\cal L}}
\end{xy}\]
\begin{definition}
\label{def:L-closure}
\ 
\begin{enumerate}
\renewcommand{\theenumi}{\alph{enumi}}
\renewcommand{\labelenumi}{\rm(\theenumi)}
	\item Let $L$ and $M$ be preordered sets, $M$ complete from below, $K\subseteq L$ and
$\varphi:K\longrightarrow M$ isotonic. 
	 Then $K^* := \lbrace f\in L\,|\, \exists\, b\in K: f\le b \rbrace$ and 
	$\varphi^*\!: K^*\longrightarrow M \,,\, f\longmapsto 
	\inf \lbrace\, \varphi(b) \,|\, f \le b\in K  \,\rbrace$\,. 
	\item Let $L$ be a set, ${\cal L}\subseteq{\cal P}(L)$ complete from below (with respect to $\subseteq$) and $K\in{\cal L}^*$. 
	 The set $id_{\cal L}^{\,*}(K) = \inf \lbrace\, B\in{\cal L} \,|\, K\subseteq B \,\rbrace$ is called {\em $\cal L$-closure of $K$}\index{closure}\,. 
	\item Let $K,L,M,\varphi$ and $\cal L$ be as in (a) and (b) and $id_{\cal L}^{\,*}(K) \subseteq K^*$. Then the restriction of $\varphi^*$ to $id_{\cal L}^{\,*}(K)$ is called {\em $\cal L$-continuation of $\varphi$}\index{continuation}\,. 
\end{enumerate}
\end{definition}

\begin{theorem}
\label{satz:L-closure}
Let $K,L,M$ and $\varphi$ be as in Definition \ref{def:L-closure} (a).
Then $\varphi^*$ is an isotonic continuation of $\varphi$ on $K^*$.
\end{theorem}

\begin{proof} 
$\varphi^*$ is a continuation: 
$K\subseteq K^*$ and for all $f\in K$ we have
$\varphi(f) \le \lbrace\, \varphi(b) \,|\, f\le b\in K \,\rbrace \owns \varphi(f)$\,,
so 
$\varphi(f) \le \varphi^*(f) \le \varphi(f)$\,,
so $\varphi^*(f)=\varphi(f)$\,, because $M$ is complete from below.
\\
$\varphi^*$ is isotonic:
Let $f,g\in K^*$ and $f\le g$\,. 
Then $\lbrace\, \varphi(b) \,|\, g\le b\in K \,\rbrace \subseteq 
\lbrace\, \varphi(b) \,|\, f\le b\in K \,\rbrace$\,, so 
$\varphi^*(f) \le \varphi^*(g)$\,. 
\end{proof}

\begin{example}
\label{exa:L-closure}
\begin{enumerate}
\renewcommand{\theenumi}{\alph{enumi}}
\renewcommand{\labelenumi}{\rm(\theenumi)}
	\item The mapping $\varphi^*$ in Definition \ref{def:L-closure} (a) is the $\lbrace K^* \rbrace$-continuation of $\varphi$\,. 
The mapping $\varphi^*$ in Theorem \ref{satz:phiAbschluss} is the $\lbrace \overline K \rbrace$-continuation of $\varphi$\,. 
	\item Let $L$ be a topological space and $\cal L$ the set of all closed subsets of $L$\,. Then ${\cal L}^*={\cal P}(L)$ and the $\cal L$-closure of $K\in{\cal P}(L)$ is the usual closure in topology. 
	\item Let $L$ be a group and $\cal L$ the set of all subgroups of $L$\,. Then ${\cal L}^*={\cal P}(L)$ and the $\cal L$-closure of $K\in{\cal P}(L)$ is the subgroup of $L$ generated by $K$\,. 
	\item Let $\Omega$ be a set and $\cal L$ the set of all $\sigma$-algebras on $\Omega$\,. Then ${\cal L}^*={\cal P}({\cal P}(\Omega))$ and the $\cal L$-closure of $K\in{\cal P}({\cal P}(\Omega))$ is the $\sigma$-algebra generated by $K$\,. 
\end{enumerate}
\end{example}

The construction of an outer measure and a measure with the help of a pre-measure is also a special case of Definition \ref{def:L-closure}:

\begin{example}
\label{exa:measure}
Let $\Omega$ be a set, $R$ a ring of subsets of $\Omega$ and $\mu: R\longrightarrow \lbrack 0 \,, \infty \rbrack$ a pre-measure.
It might be conceivable to extend $\mu$ through
$\mu^*\left( \bigcup\limits_{i=1}^\infty A_i \right) := \sum\limits_{i=1}^\infty \mu(A_i)$ for all sequences $(A_i)_i$ of pairwise disjoint elements of $R$\,. But then it is questionable whether $\mu^*$ is well-defined or not. So we choose an other way:

Let $K$ respectively $L$ be the set of all sequences of pairwise disjoint elements of $R$ respectively ${\cal P}(\Omega)$\,, 
$\varphi: K\longrightarrow \lbrack 0 \,, \infty \rbrack$\,, 
$f \longmapsto \sum_i \mu(f_i)$ 
and $f\le g \,:\Leftrightarrow\, \bigcup_i f_i\subseteq\bigcup_i g_i$ for $f,g\in L$\,. 
Then $L$ and $\lbrack 0 \,, \infty \rbrack$ are preordered, $\lbrack 0 \,, \infty \rbrack$ is complete from below and $\varphi$ is isotonic: 
$\sum\limits_{i=1}^n \mu(f_i) = \mu\left( \bigcup\limits_{i=1}^n f_i \right) = 
\mu\left( \bigcup\limits_{j=1}^\infty \bigcup\limits_{i=1}^n f_i\cap g_j \right) = 
\sum\limits_{j=1}^\infty \mu\left( \bigcup\limits_{i=1}^n f_i\cap g_j \right) \le 
\sum\limits_{j=1}^\infty \mu\left( g_j \right) = \varphi(g)$ 
for all $n\in \N$ and $f,g\in K$\,, $f\le g$\,. 
Every assumption of Definition \ref{def:L-closure} is fulfilled. 
We make the additional assumption that there is an $h\in K$ with $\bigcup_i h_i = \Omega$\,. Then $K^*=L$\,. 
Let $\iota: {\cal P}(\Omega) \hookrightarrow L$\,, $S\mapsto (S,\emptyset,\emptyset,\dots)$\,. Then $\mu^*:= \varphi^*\circ\iota$ is an outer measure and a continuation of $\mu = \varphi\circ\iota\big|_R$ (cf. \cite{Bau} I.5). 
Let $\cal L$ be the set of all $\sigma$-algebras on $\Omega$\,. Then the restriction of $\mu^*$ to the $\cal L$-closure $id_{\cal L}^{\,*}(R)$ is a measure on the $\sigma$-algebra $id_{\cal L}^{\,*}(R)$\,. 
\end{example}

\end{document}